\documentclass{article}
\usepackage{anyfontsize}
\usepackage[margin=1.25in]{geometry}
\usepackage{hyperref}
\usepackage{times,mathrsfs,mathtools,amsmath,amssymb}

\usepackage{amsthm}
\usepackage{stix}
\usepackage{titlesec}

\usepackage{parskip}

\begin{document}

\newtheorem{theorem}{\textsc{Theorem}}[section]
\newtheorem{problem}[theorem]{\textsc{Problem}}
\newtheorem{exercise}{\textsc{Exercise}}[section]
\newtheorem{proposition}[theorem]{\textsc{Proposition}}
\newtheorem{lemma}[theorem]{\textsc{Lemma}}
\newtheorem{corollary}[theorem]{\textsc{Corollary}}
\newtheorem{definition}[theorem]{\textsc{Definition}}
\newtheorem{remark}[theorem]{\rm \textsc{Remark}}
\newtheorem{example}[theorem]{\rm \textsc{Example}}

\renewcommand{\qedsymbol}{$\blacksquare$}

\newcommand{\bb}{\mathbb}
\newcommand{\mc}{\mathcal}
\renewcommand{\bf}{\mathbf}
\renewcommand{\bar}{\overline}
\renewcommand{\Re}{\text{Re}\,}
\renewcommand{\Im}{\text{Im}\,}
\newcommand{\im}{\text{im}\,}
\newcommand{\wtilde}{\widetilde}
\newcommand{\what}{\widehat}
\newcommand{\rhu}{\rightharpoonup}
\newcommand{\la}{\langle}
\newcommand{\ra}{\rangle}
\renewcommand{\r}{\right}
\renewcommand{\l}{\left}
\newcommand{\ind}{\text{ind}}
\newcommand{\res}{\text{Res}}
\newcommand{\bs}{\boldsymbol}
\newcommand{\tx}{\text}
\renewcommand{\v}{\tx{\bf v}}
\renewcommand{\u}{\tx{\bf u}}
\newcommand{\n}{\tx{\bf n}}
\newcommand{\w}{\tx{\bf w}}
\renewcommand{\div}{\text{div}\,}
\newcommand{\bdot}{\bs{\cdot}}
\newcommand{\on}{\operatorname}
\newcommand{\hooklongrightarrow}{\lhook\joinrel\longrightarrow}

\setcounter{section}{0}

\title{A Unique Continuation Property for the Level Set Equation}
\author{Nick Strehlke\footnote{This research was partially supported by NSF Grant DMS 1404540.}}
\date{}

\maketitle

\begin{abstract}
We prove the following unique continuation result: if a solution to the level set equation for mean curvature flow in a mean-convex domain agrees to infinite order at the point where it attains its maximum with the solution for a ball, then it agrees everywhere and the domain is a ball.
\end{abstract}

\section{Introduction}

Let $\Omega$ be a mean-convex domain in $\mathbb R^{n+1}.$ The level set equation for mean curvature flow is a degenerate elliptic boundary value problem asking for a function $t\colon \Omega \to\mathbb R$ satisfying $t=0$ on $\partial \Omega$ and
\begin{align}
	|\nabla t| \on{div}\left(\frac{\nabla t}{|\nabla t|}\right) = -1. \label{lse}
\end{align}
The solution $t$ to this problem exists and is unique and twice differentiable on $\Omega.$ It is sometimes called the \emph{arrival time} for mean curvature flow on the domain $\Omega,$ for the following reason: Let $T = \max_{x\in\Omega} t(x)$ and for $\tau\in [0,T)$ define the hypersurface $M_\tau$ by
\begin{align*}
	M_\tau &= \{x\in \Omega\colon t(x) = \tau\}.
\end{align*}
Then the $1$-parameter family $\{M_\tau\}_{\tau\in [0,T)}$ of surfaces is a mean curvature flow: the position vector $x(\tau)$ of $M_\tau$ satisfies the equation
\begin{align*}
	\frac{\mathrm d x}{\mathrm d\tau} \bdot \bs N(x,\tau) &= - H(x,\tau),
\end{align*}
where $\bdot$ is the inner product in $\mathbb R^{n+1}$ and $H$ and $\bs N$ are the scalar mean curvature and outer unit normal, respectively, of the hypersurface $M_\tau.$\footnote{This point of view for studying mean curvature flow was first taken in a computational context by Osher and Sethian, \cite{OsSe88}, and the theory was then developed in \cite{ChGi91} and \cite{EvSp91}.}

When $\Omega = B_r(x_0)$ is the ball of radius $r$ centered at the point $x_0$ in $\mathbb R^{n+1},$ the solution $t_B$ to (\ref{lse}) is
\begin{align}
	t_B(x) &= \frac{r^2}{2n} - \frac{|x-x_0|^2}{2n} \label{lse-ball}
\end{align}
and the corresponding mean curvature flow is a family of homothetically shrinking spheres. The main result is that a solution to the level set equation (\ref{lse}) on a mean-convex domain $\Omega$ that attains its maximum at $x_0$ and agrees to infinite order at the point $x_0$ with the solution $t_B$ to the level set equation for a ball centered at $x_0$ must actually coincide with $t_B$ everywhere.\footnote{A different and more complicated \emph{parabolic} unique continuation property for self-shrinkers was recently proved by Jacob Bernstein in \cite{Be17}.}

\medskip
\begin{theorem}
	\label{unique-cont} Suppose $\Omega\subset\mathbb R^{n+1}$ is a mean-convex domain and $t\colon \Omega \to\mathbb R$ solves (\ref{lse}) and attains its maximum $T$ at the point $x_0\in \Omega.$ If
	\begin{align}
		t(x) = T - \frac{|x-x_0|^2}{2n} +O\left( |x-x_0|^N\right) \label{asympt}
	\end{align}
	as $x\to x_0,$ for every integer $N> 2,$ then $t(x) = T - |x - x_0|^2/(2n)$ and $\Omega = B_{(2nT)^{1/2}}(x_0)$ is the ball of radius $(2n T)^{1/2}$ centered at $x_0.$
\end{theorem}

\emph{Remarks.} 

1. Let $\Omega$ be a mean-convex domain and suppose $t\colon \Omega\to\mathbb R$ satisfies (\ref{lse}). We already mentioned that $t$ is twice differentiable on its domain (this was proved in \cite{CoMi16}).\footnote{It was shown in \cite{Wh00} (Theorem 1.2) and \cite{Wh11} that any tangent flow of a smooth mean convex mean curvature flow is a generalized cylinder. From this one can figure out what the Hessian of the arrival time function must be if it exists. The remaining issue was to show that the Hessian exists, which is equivalent to the problem of uniqueness of tangent flows. This was solved in \cite{CoMi15}.} The Hessian of $t$ at a critical point $x_0$ is then
\begin{align*}
	\nabla^2 t(x_0) &= - \frac{1}{2k}P,
\end{align*}
where $k$ is an integer between $1$ and $n$ and $P$ is a projection onto a $n-k$ dimensional hyperplane. See \cite{CoMi18}. The hypothesis of Theorem \ref{unique-cont} implies that $k=n.$ In this case $t$ is actually $C^2$ in a neighborhood of the maximum, and the corresponding mean curvature flow becomes extinct at $x_0$ in such a way that the rescaled mean curvature flow converges to a round sphere. In particular, the isolated point $\{x_0\}$ is a connected component of the level set $\{x\colon t(x) = T\},$ and nearby level sets are convex. Because a mean curvature flow cannot coincide with a sphere at any time unless it is a shrinking sphere, it is therefore sufficient to prove Theorem \ref{unique-cont} in case $\Omega$ is a convex domain.

2. The hypothesis (\ref{asympt}) for a fixed $N>2$ likely implies (for a solution to (\ref{lse})) that $t$ is $C^{N-1}$ near $x_0.$ This is known in case $N=3$ (Theorem 6.1 of \cite{Hu93}) and in case $N=4$ (Corollary 5.1 of \cite{Se08}). This fact is not required for our result, however, and we do not investigate it here. Of course, Theorem \ref{unique-cont} implies that $t$ is analytic if it satisfies (\ref{asympt}) for \emph{all} $N.$

\section{Proof of Theorem \ref{unique-cont}}

To prove Theorem \ref{unique-cont}, we relate the asymptotic behavior of $t$ near its maximum to the behavior of the corresponding mean curvature flow near its singularity. We show that the hypothesis (\ref{asympt}) for all $N>2$ implies that the rescaled mean curvature flow converges to a stationary round sphere at a rate that is faster than any exponential, and then that this cannot happen unless the rescaled flow is identically equal to the stationary sphere. As mentioned in the remarks following the theorem, it is sufficient to prove Theorem \ref{unique-cont} in case $\Omega$ is a convex domain, and we restrict attention to this case from here on.

We now briefly describe the rescaled mean curvature flow for a convex surface. Let $\Omega\subset\mathbb R^{n+1}$ be a bounded convex region and let $\{M_\tau\}_{\tau\in [0,T]}$ be the mean curvature flow starting from $M_0 = \partial \Omega.$ By a theorem of Huisken, \cite{Hu84}, the flow $M_t$ shrinks down to a single point $x_0$ (that is, $M_T =\{x_0\}$) and the translated and rescaled flow $(T-\tau)^{-1/2}(M_\tau- x_0)$ converges in $C^k,$ for any $k,$ to the round sphere $\bf{S}^n$ of radius $(2n)^{1/2}$ centered at the origin.

The following rescaling procedure is standard in the study of mean curvature flow.\footnote{See the first sections of \cite{CoMi12} for an excellent introduction and overview.} We perform the substitution $s = -\log{(T-\tau)}$ and denote by $\Sigma_s$ the surface $(T-\tau)^{-1/2}(M_\tau - x_0).$ The $1$-parameter family $\{\Sigma_s\}_{s\in [-\log{T},\infty)}$ is said to be a \emph{rescaled mean curvature flow}. By rescaling the surfaces $M_\tau,$ we can arrange that the time $T$ of extinction is $1,$ and therefore that the rescaled flow is defined for $s\geq 0.$ Its position vector $y(s)$ satisfies the equation
\begin{align*}
	\frac{\mathrm d y}{\mathrm ds} \bdot N(y,s) = - H(y,s) + \frac{1}{2}y\bdot N(y,s),
\end{align*}
where $H$ and $N$ are the scalar mean curvature and outer unit normal, respectively, for $\Sigma_s.$

Let $\bf{n}(x) = x/|x| = x/(2n)^{1/2}$ be the outer unit normal to the sphere $\bf{S}^n$ at the point $x.$ Because $\Sigma_s$ converges in $C^k$ to $\bf{S}^n$ as $s\to\infty,$ the surface $\Sigma_s$ is a normal graph over $\bf{S}^n$ for sufficiently large $s$: there exists $s_0\geq0$ and a scalar function $u\colon \bf{S}^n\times [s_0,\infty)\to\mathbb R$ for which
\begin{align}
	\Sigma_s &= \{x+ u(x,s)\bf{n}(x)\colon x\in\bf{S}^n\}. \label{defn.u}
\end{align}
The function $u$ is uniquely determined as the solution of a quasilinear parabolic PDE on the sphere. The parabolic evolution has the constantly zero function for a solution (a stationary state), corresponding to the fact that the stationary sphere $\bf{S}^n$ satisfies the rescaled MCF equation. The linearization of the parabolic operator at the zero solution is $\partial_s - \Delta - 1,$ where $\Delta$ is the Laplacian on the sphere. The following lemma converts Theorem \ref{unique-cont} to a problem about functions $u$ satisfying this parabolic equation: if $t$ is a solution to the level set equation satisfying the assumption (\ref{asympt}) for all $N,$ then the corresponding solution to rescaled MCF converges to $\bf{S}^n$ faster than any exponential. The problem is then reduced to showing that this cannot happen unless the rescaled MCF is identically $\bf{S}^n.$

Note we may assume by translating everything that $x_0 = 0.$ Thus we state the lemma for the case when the mean curvature flow becomes extinct at the origin.

\medskip
\begin{lemma}
	\label{conv} Suppose $\Omega\subset\mathbb R^{n+1}$ is a convex domain and $t\colon \Omega \to\mathbb R$ solves (\ref{lse}) and attains its maximum $T$ at the origin. Suppose
	\begin{align}
		t(x) = T - \frac{|x|^2}{2n} +O\left( |x|^N\right) \label{asympt'}
	\end{align}
	as $x\to 0,$ for every integer $N> 2.$ Then for any integer $k\geq 0,$ the rescaled MCF $\{\Sigma_s\}$ that $t$ defines converges to $\bf{S}^n$ in the Sobolev space $H^k$ faster than any exponential. 
\end{lemma}
\begin{proof}
We use the following fact: if $\{\Sigma_s\} = \{x+ u(x,s)\bf{n}(x)\colon x\in\bf{S}^n\}$ is a rescaled MCF converging as $s\to\infty$ to the sphere in $L^2,$ that is, for which $u$ converges to zero in $L^2(\bf{S}^n),$ then $u$ is bounded in $H^k(\bf{S}^n)$ for every $k\geq 1$ (actually $u$ converges exponentially in $H^k$). On account of the interpolation inequalities 
\begin{align*}
	\|u\|_{H^k}\leq \|u\|_{L^2}^{1/2} \|u\|_{H^{2k}}^{1/2},
\end{align*}
this reduces the proof of the lemma to showing that the function $u$ converges to zero faster than any exponential in $L^2(\bf{S}^n).$ 

Now we show that (\ref{asympt'}) holding for all $N$ implies that $u(s)$ converges to zero faster than any exponential in $L^\infty(\bf{S}^n),$ hence in $L^2(\bf{S}^n).$ 

First observe: if (\ref{asympt'}) holds for some $N,$ then 
\begin{align*}
	|x|^2 = (2n)(T-t(x)) + O\left((T-t(x))^{N/2}\right)
\end{align*}
as $T-t\to0,$ for the same $N.$ Next, use this to write
\begin{align*}
	(T-t(x))^{-1/2}x  &= \left( |x|^2/(2n) + O\left((T-t(x))^{N/2}\right) \right)^{-1/2}x  \\
	&= (2n)^{1/2} \frac{x}{|x|} + O\left( \frac{(T- t(x))^{N/2}}{|x|^2}\right) \frac{x}{|x|} \\
	&= (2n)^{1/2} \frac{x}{|x|} + O\left((T-t(x))^{N/2-1}\right)\frac{x}{|x|}.
\end{align*}
This says exactly that if $s = -\log{(T-\tau)}$ and
\begin{align*}
	\{y+ u(y,s) \bf{n}(y)\colon y\in \bf{S}^n\} = \Sigma_s = \left\{ (T-\tau)^{-1/2} x\colon t(x) = \tau\right\}
\end{align*}
is the rescaled MCF defined by $t$ (it will be a graph over $\bf{S}^n$ for $\tau$ sufficiently close to $T$), then
\begin{align*}
	|u(y,s)| = O\left((T-\tau)^{N/2 - 1}\right) = O\left(e^{-(N/2 - 1) s}\right)
\end{align*}
as $s \to\infty.$ Therefore if (\ref{asympt'}) holds for all $N,$ then $u$ converges to zero faster than any exponential in $L^\infty(\bf{S}^n).$ Since $\bf{S}^n$ has finite volume, $u$ converges to zero faster than any exponential in $L^2(\bf{S}^n)$ as well.
\end{proof}

Having established Lemma \ref{conv}, Theorem \ref{unique-cont} is a consequence of the following theorem.

\medskip
\begin{theorem}
	\label{rapid} Let $r>n/2+1$ be an integer. Suppose $\Sigma_s = \{y+u(y,s)\bf{n}(y)\colon y\in \bf{S}^n\},$ $s\geq 0,$ is a rescaled MCF and suppose that it converges to $\bf{S}^n$ in $H^{r}(\bf{S}^n)$ faster than any exponential in the sense that
	\begin{align*}
		\lim_{s\to\infty} e^{\sigma s}\| u(\cdot,s)\|_{H^{r}(\bf{S}^n)} = 0
	\end{align*}
	for all $\sigma>0.$ Then $u$ is identically zero and $\Sigma_s = \bf{S}^n$ for all $s.$
\end{theorem}

The equation satisfied by $u$ in order for the normal graph $\Sigma_s = \{y+u(y,s)\bf{n}(y)\colon y\in \bf{S}^n\}$ to evolve by rescaled MCF can be written in the form
\begin{align}
	\partial_s u &= \Delta u + u + N(u,\nabla u,\nabla^2u) \label{rMCF'}
\end{align}
where $\Delta$ is the Laplacian on $\bf{S}^n$ and $N$ is a nonlinear term of the form
\begin{align}
	N(u,\nabla u,\nabla^2 u) = f(u,\nabla u) + \on{trace}(B(u,\nabla u)\nabla^2u), \label{nonlin-error}
\end{align}
where $f$ and $B$ are smooth, $B(0,0) = f(0,0) = \mathrm df(0,0) = 0.$ In other words, $N$ vanishes up to quadratic error at zero. 

The important feature of this equation is that the linear operator $\partial_s - \Delta -1$ gives a good approximation to the nonlinear operator in (\ref{rMCF'}): in a Sobolev space $H^r(\bf{S}^n)$ of high enough order $r,$ a function $u$ the normal graph of which evolves by rescaled MCF satisfies
\begin{align*}
	\|(\partial_s - \Delta - 1)u\|_{H^r(\bf{S}^n)} \leq C \|u\|_{H^{r+1}(\bf{S}^n)} \|u\|_{H^{r+2}(\bf{S}^n)}.
\end{align*}
This bound is what implies that if $u$ converges to zero as $s\to+\infty,$ it must do so at an exponential rate (unless it is identically zero). We state this bound as a lemma. This is a textbook result, but the proof is also written in a companion paper.

\medskip
\begin{lemma}
\label{nonlin-bd}
	Suppose $N$ is a smooth function satisfying (\ref{nonlin-error}). If $r>n/2+1$ is an integer and $\|u\|_{H^r(\bf{S}^n)}\leq 1,$ then 
	\begin{align*}
		\|N(u,\nabla u,\nabla^2 u)\|_{H^r(\bf{S}^n)} \leq C \|u\|_{H^{r+1}(\bf{S}^n)} \|u\|_{H^{r+2}(\bf{S}^n)}
	\end{align*}
	for some constant $C$ depending only on $N$ and $r.$
\end{lemma}

Using Lemma \ref{nonlin-bd}, we now prove Theorem \ref{rapid}.

\begin{proof}[Proof of Theorem \ref{rapid}]
	Throughout the proof, we abbreviate the $H^r(\bf{S}^n)$ norm $\|\cdot\|_{H^r(\bf{S}^n)}$ by $\|\cdot\|_r,$ we abbreviate $u(\cdot,t)$ by $u(t),$ and we abbreviate the nonlinear error $N(u,\nabla u,\nabla^2 u)$ by $N(u).$  Let $\lambda_1<\lambda_2<\cdots$ be the eigenvalues of the operator $-\Delta - 1,$ and denote by $\Pi_k$ orthogonal projection onto the direct sum of eigenspaces corresponding to $\lambda_j$ with $j\geq k.$ 
	
	We will prove that a solution $u$ satisfying (\ref{rMCF'}) obeys, for each positive integer $k$ and each $s_0\geq 0,$ the inequality
\begin{align}
	e^{\lambda_k (s-s_0)}\|u(s)\|_r \leq \|\Pi_{k} u(s_0)\|_r + \int_{s_0}^\infty e^{\lambda_k (t-s_0)} \|N(u(t))\|_r\,\mathrm dt. \label{arrival.0th.1}
\end{align}
Lemma \ref{nonlin-bd} implies that there is a constant $C$ depending on $N$ with the property that
\begin{align*}
	\|N(u)\|_r \leq C\|u\|_{r+1}\|u\|_{r+2}\leq C\|u\|_{r+2}^2\leq C\|u\|_r\|u\|_{r+4}.
\end{align*}
The last inequality follows from Cauchy--Schwarz and integration by parts for example.

It then follows that
\begin{align*}
	e^{\lambda_k (s-s_0)}\|u(s)\|_r &\leq \|\Pi_{k} u(s_0)\|_r + C\int_{s_0}^\infty e^{\lambda_k (t-s_0)} \|u(t)\|_{r}\|u(t)\|_{r+4}\,\mathrm dt \\
	&\leq  \|\Pi_{k} u(s_0)\|_r + C\left(\sup_{t\geq s_0}e^{\lambda_k (t-s_0)} \|u(t)\|_r\right) \int_{s_0}^\infty \|u(t)\|_{r+4}\,\mathrm dt. 
\end{align*}
Taking the supremum over $s\geq s_0$ on the left side then gives
\begin{align*}
	\left(\sup_{t\geq s_0}e^{\lambda_k (t-s_0)} \|u(t)\|_r\right)\left(1 - C\int_{s_0}^\infty \|u(t)\|_{r+4}\,\mathrm dt\right)\leq \|\Pi_{k} u(s_0)\|_r
\end{align*}
for all $k.$ Because convergence to the sphere is necessarily exponential in $C^k$ for every $k,$ the $H^{r+4}(\bf{S}^n)$ norm $\|u(t)\|_{r+4}$ is integrable (as are all other $H^s$-norms), and therefore we can choose $s_0$ so large that
\begin{align}
	C\int_{s_0}^\infty \|u(t)\|_{r+4}\,\mathrm dt\leq 1/2. \label{small-int}
\end{align}
Moreover, we choose $s_0$ to be the least nonnegative number for which (\ref{small-int}) holds. For this $s_0$ we obtain
\begin{align*}
	\sup_{t\geq s_0}e^{\lambda_k (t-s_0)} \|u(t)\|_r\leq 2 \|\Pi_{k} u(s_0)\|_r
\end{align*}
for all positive integers $k,$ and since the right side vanishes in the limit $k\to\infty$ and the left side is non-decreasing in $k$ it follows that $\|u(t)\|_r = 0$ for $t\geq s_0.$ 

Now we show that $s_0=0.$ Since $\|u(t)\|_r = 0$ for $t\geq s_0,$ we also have $\|u(t)\|_{r+4} = 0$ for $t\geq s_0.$ Consequently,
\begin{align*}
	C\int_{s_0}^\infty \|u(t)\|_{r+4}\,\mathrm dt = 0.
\end{align*}
If $s_0>0,$ then it cannot possibly be the smallest positive number for which (\ref{small-int}) holds, and we arrive at a contradiction. Thus $u$ is in fact identically zero for all $s\geq 0.$

Therefore it is enough to prove (\ref{arrival.0th.1}). Write $L = \Delta+1$ for brevity. We will briefly explain how the assumption that $u$ converges to zero faster than any exponential leads to the following representation formula:
\begin{align*}
	e^{\lambda_k (s-s_0)} u(s) & = e^{(\lambda_k + L)(s-s_0)} \Pi_{k} u(s_0)   + \int_{s_0}^s e^{(\lambda_k + L)(s-t)} e^{\lambda_k(t-s_0)} \Pi_{k} N(u(t))\,\mathrm dt \\
	&\quad  - \int_s^\infty e^{(\lambda_k + L)(s-t)}e^{\lambda_k (t-s_0)}(1-\Pi_{k}) N(u(t))\,\mathrm dt,
\end{align*}
 Notice that $\lambda_k +L$ is non-positive definite on the range of $\Pi_{k}$ and positive definite on the range of $1-\Pi_{k}.$ Thus (\ref{arrival.0th.1}) follows by simply taking the $H^r$ norm of both sides and applying the triangle inequality repeatedly to the right side. 
 
 The representation can be derived from the variation of constants formula
\begin{align*}
	u(s_1)  &= e^{L(s_1-s)} u(s) + \int_{s}^{s_1} e^{L(s_1-t)} N(u(t))\,\mathrm dt,
\end{align*}
where $s_1\geq s,$ in the following way.\footnote{This is a standard trick in the construction of invariant manifolds for ODE.} Apply the projection $1-\Pi_k$ to both sides of the variation of constants formula. Apply $e^{-L(s_1-s)},$ which is a bounded operator on the finite-dimensional image of $1-\Pi_k,$ to both sides and rearrange to obtain
\begin{align*}
	(1-\Pi_k)u(s) &= e^{-L(s_1-s)} (1-\Pi_k)u(s_1) - \int_{s}^{s_1} e^{L(s-t)}(1-\Pi_k) N(u(t))\,\mathrm dt.
\end{align*}
Since $e^{-L(s_1 - s)}$ is bounded by $e^{\lambda_{k-1}} (s_1- s)$ on the image of $1-\Pi_k,$ we can send $s_1\to\infty$ and by our assumption that $u_1-u_2$ vanishes more rapidly than any exponential, the term $e^{-L(s_1 - s)} (1-\Pi_k) u(s_1)$ converges to zero. The integral, meanwhile, is absolutely convergent. Thus we obtain
\begin{align*}
	(1-\Pi_k)u(s) &= - \int_s^\infty e^{L(s-t)}(1-\Pi_k) N(u(t))\,\mathrm dt,
\end{align*}
for any $k.$ Substituting this equation into the variation of constants formula (with $s_1$ replaced by $s$ and $s$ replaced by $s_0$) gives the representation described above.
\end{proof}

\bibliography{/Users/nstrehlke/Documents/MathBibliography}
\bibliographystyle{amsalpha}

\end{document}